\documentclass[11pt]{amsart}
\usepackage{amsmath,latexsym,amssymb,verbatim}
\usepackage{graphicx}
\usepackage{psfrag}
\usepackage{tikz}
\usepackage{hyperref}

\newtheorem{lema}{Lemma}[section]
\newtheorem{theo}[lema]{Theorem}
\newtheorem{prop}[lema]{Proposition}
\newtheorem{coro}[lema]{Corollary}
\theoremstyle{definition}

\theoremstyle{remark}
\newtheorem{rema}[lema]{Remark}

\newcommand{\disc}{\operatorname{Disc}}

\newcommand{\sgn}{\operatorname{sgn}}

\definecolor{rojo}{rgb}{0.75,0,0}
\definecolor{oro}{rgb}{0.85, 0.65, 0.13}
\definecolor{azul}{rgb}{0.05, 0.05, 0.93}

\title[Stability of singular limit cycles]{Stability of singular limit cycles\\ for Abel equations revisited}
\author{J.L. Bravo, M. Fern\'{a}ndez, I. Ojeda}

\address{J.L. Bravo, Departamento de Matem\'{a}ticas,
Universidad de Extremadura, 06006 Badajoz, Spain}
\email{trinidad@unex.es}

\address{M. Fern\'{a}ndez, Departamento de Matem\'{a}ticas, Universidad de Extremadura, 06006 Badajoz, Spain}
\email{ghierro@unex.es}

\address{I. Ojeda, Departamento de Matem\'{a}ticas, Universidad de Extremadura, 06006 Badajoz, Spain }
\email{ojedamc@unex.es}


\begin{document}

\begin{abstract}
A criterion is obtained for the semi-stability of the isolated
singular positive closed solutions, i.e., singular positive limit
cycles, of the Abel equation $x'=A(t)x^3+B(t)x^2$, where $A,B$ are
smooth functions with two zeros in the interval $[0,T]$ and where
these singular positive limit cycles satisfy certain conditions, which
allows an upper bound on the number of limit cycles of the Abel
equation to be obtained.  The criterion is illustrated by obtaining
an upper bound of two positive limit cycles for the family
$A(t)=t(t-t_A)$, $B(t)=(t-t_B)(t-1)$, $t\in[0,1]$.  In the linear
trigonometric case, i.e., when $A(t)=a_0+a_1\sin t +a_2\cos t$,
$B(t)=b_0+b_1\sin t+b_2 \cos t$, an upper bound of two limit cycles is
also obtained for $a_0,b_0$ sufficiently small and in the
region where two positive limit cycles bifurcate from the origin.
\end{abstract}

\maketitle

\noindent
\textcopyright{} 2023 The Authors.
This manuscript version is made available under the
\href{https://creativecommons.org/licenses/by-nc-nd/4.0/}
{CC BY-NC-ND 4.0 license}.

\noindent
Published version:
\href{https://doi.org/10.1016/j.jde.2023.10.003}
{doi:10.1016/j.jde.2023.10.003}.

\section{Introduction and Main Results}

We consider Abel equations
\begin{equation}\label{eq:Abel}
\frac{dx}{dt}=x'=A(t)x^3+B(t)x^2,
\end{equation}
with $A,B$ smooth functions defined on $[0,T]$. Let $u(t,x)$ denote
the solution of \eqref{eq:Abel} determined by $u(0,x)=x$. We say
$u(t,x)$ is closed or periodic if $u(T,x)=x$. Let $u(t,x)$ be
closed. It is singular or multiple if $u_x(T,x)=1$, otherwise it is
simple or hyperbolic.  Isolated closed solutions are also called limit
cycles. A singular closed solution such that $u_{xx}(T,x)\ne0$ is
called a double closed solution or a semistable limit cycle. The
problem of determining the maximum number of closed solutions of
\eqref{eq:Abel} is the ``Pugh problem'' mentioned by Smale~\cite{S}.

Notice that $x=0$ is always a closed solution of
\eqref{eq:Abel}. Therefore the number of closed solutions in regions
$x>0$ and $x<0$ can be studied separately. Since one region can be
translated to the other with the transformation $x\rightarrow -x$, we
shall restrict attention to the region $x>0$.

There are several results for uniqueness of closed solutions
of~\eqref{eq:Abel} on $x>0$. The best known impose that one of the
functions $A$ or $B$ does not change sign (see \cite{GG,GL,Ll,Panov,
  Pliss}). Other conditions, allowing $A$ and $B$ to change sign, are
considered for instance in \cite{AGG,BFG}. In all these results, the
condition of a definite sign is imposed on a certain derivative of the
return map  or on the initial conditions corresponding
to positive closed solutions. Applying these results
to~\eqref{eq:Abel}, one determines families for which there is at most
one positive closed solution.

A different approach is taken in \cite{BFG2} where, in order to obtain
 two positive closed solutions as upper bound, the Abel equation is
considered to be a member of a one-parameter family,
\begin{equation}\label{eq:Abelparametric}
x'=A(t,\lambda)x^3+B(t,\lambda)x^2, \quad \lambda\in\mathbb{R}
\end{equation}
where $F(t,x,\lambda):=A(t,\lambda)x^3+B(t,\lambda)x^2$ satisfies
$F_\lambda(t,x,\lambda)> 0$ for $x>0$. Thus, $\lambda \to
F(t,x,\lambda)$ is strictly increasing for all $t\in\mathbb{R}$,
$x>0$. This is termed monotonic with respect to $\lambda$.

Notice that the above definition of monotonic with respect to
$\lambda$ for families of Abel equations is an adaptation of the
setting of the so-called rotated families of planar vector fields
introduced by G.F.D. Duff, see \cite{Duff} or
\cite[Sec. 4.6]{Perko}. For these families of vector fields, the
control of bifurcations of double closed solutions is crucial to
understanding their global bifurcation diagram of closed solutions.

We consider simple Abel equations for which there is no uniqueness of
positive closed solutions, and study their number by controlling the
nature of the double closed solutions. In~\cite{BFG2}, we studied the
case where $A$ has two simple zeros of which one is at $t=0$, and $B$ has
one simple zero in $[0,T]$. In the present work, we consider the case
where $B$ has two simple zeros in $[0,T]$.  Our main result provides
sufficient conditions to determine the stability of positive singular
closed solutions.

Throughout this communication we shall write
\begin{equation}\label{eq:upper}
	P(t)=4(B (t)A '(t)-B '(t)A (t)) - B ^3(t)
\end{equation}
and
\begin{equation}\label{ecu:v}
v(t,x)=B(t)(2A(t)x+B(t))^2+P(t).
\end{equation}

\begin{theo}\label{theo:bif}
If
	\begin{itemize}
		\item[$(C_1)$]  $A(0)=0$, $A(t)$  has a simple zero  $t_A \in (0,T)$ and $B(t)$ has two simple zeros  $ t_{B_1},t_{B_2}\in [0,T]$ with $0<t_{B_1}<t_A < t_{B_2}\leq T,$
	\end{itemize}
	and for any positive singular closed solution  $\tilde u(t):=u(t,\tilde x)$ of Abel equation \eqref{eq:Abel}
	\begin{itemize}
		\item[$(C_2)$] the function $2A(t)\tilde u(t)+B(t)$ has at most a simple zero in 
		each of the intervals $[0,t_A]$ and $[t_A,T]$, 
		\item[$(C_3)$] $\sgn\left(v(t,\tilde u(t))\right) = \sgn(A'(0)B(0)),$ for all $t\in[0,T],$
	\end{itemize}
	then $ u_{xx}(T,\tilde x)=\sgn(A'(0)B(0))$.
\end{theo}

\begin{rema}\label{rem:signos}
For the sake of simplicity of exposition, we assume $A'(0)<0$ and
$B(0)>0$. So condition $(C_3)$ becomes $v(t,\tilde u(t)) < 0$ for
all $t \in [0,T]$, and the conclusion is $ u_{xx}(T,\tilde
x)<0$. The other cases are proved similarly.
\end{rema}

A difficult point for the above result to be applicable is to verify
when hypotheses $(C_2)$ and $(C_3)$ hold since they include the
unknown singular closed solution. Nevertheless, in
Propositions~\ref{prop:upper} and \ref{PropC3} and
Corollary~\ref{CorC3} below, we shall give sufficient algebraic
conditions for them to be checked computationally.

As a motivating example, consider the family of Abel equations
\begin{equation} 	\label{eq:Abelquadquad}
	x'=t(t-t_A)x^3+ (t-t_{B})(t-1)x^2, 
	\quad t_A, t_{B} \in\mathbb{R},
\end{equation}
where $t\in[0,1]$. Upper bounds of the number of positive closed
solutions of \eqref{eq:Abelquadquad} have been obtained for some
cases, as will be detailed in Section~\ref{Sect4}. As a consequence of
Theorem \ref{theo:bif} above, we prove:
\begin{theo}\label{theo:example}
Abel equation \eqref{eq:Abelquadquad} has at most two positive closed solutions, 
taking into account their multiplicities, and this upper bound is sharp.
\end{theo}

As we shall see, the existence of two positive closed solutions is due
to the fact that for $t_A=2/3$ and $t_B=1/3$ the multiplicity of the
closed solution $x=0$ is four, while generically it is two. Hence a
Hopf-like codimension-two bifurcation appears, and two positive closed
solutions bifurcate from the origin.

The main motivation for this paper was Problem~6 of \cite{G},
i.e., to obtain the maximum number of limit cycles of
the Abel equation
\begin{equation}\label{eq:Abel_trigonometric_linear}
x'= \left(a_1+a_2\sin\,t+a_3\cos\,t\right)x^3+\left(b_1+b_2\sin\,t+b_3\cos\,t\right)x^2.
\end{equation}
We address this problem in Section~\ref{Sect5}.  For this equation,
the functions $A,B$ have at most two simple zeros, and a Hopf-like
codimension-two bifurcation at $a_0=b_0=0$ proves the existence of at
least two positive limit cycles. If $A$ and $B$ have at most one
simple zero, or the simple zeros of $A$ and $B$ do not alternate, the
problem is solved in~\cite{AGG} and \cite{BFG}. We prove that
Theorem~\ref{theo:bif} explains the upper bound of two positive limit
cycles in a region where two positive limit cycles bifurcate from the
origin, giving a partial answer to Problem~6 of \cite{G}. We also
discuss the limitations of Theorem~\ref{theo:bif} in this case.

\section{Stability of Singular Closed Solutions}

In this section we prove the main result, but first we shall explain 
how the stability of the singular closed solutions determines the maximum 
number of limit cycles for \eqref{eq:Abelparametric} assuming
that this number is known for certain values of the parameter.

Assume that the family \eqref{eq:Abelparametric} satisfies $F_\lambda(t,x,\lambda)>0$
for all $t\in(0,T)$, $x>0$, and $\lambda\in (\lambda_1,\lambda_2)$. If $u(t,x,\lambda)$
denotes the solution of~\eqref{eq:Abelparametric} determined by $u(0,x,\lambda)=x$ then 
$u(t)$ is a closed solution if and only if $u(T,u(0),\lambda)=u(0)$.
As the monotonicity of $F(t,x,\lambda)$ with respect to $\lambda$ implies that of  $u(t,x,\lambda)$ when the latter is positive, the Implicit Function Theorem guarantees the existence of a $\mathcal{C}^1$ function $\Lambda$ defined by 
\[
u\left(T,x,\Lambda(x)\right)=x.
\]
Therefore, for every fixed $\lambda$, the number of positive
closed solutions, $N(\lambda)$, is the number of solutions of $\Lambda(x)=\lambda$. Note that \[\Lambda'(x)=\frac{1-u_{x}(T,x,\Lambda(x))}{u_{\lambda}(T,x,\Lambda(x))},\]
where $u_{\lambda}(T,x,\Lambda(x))>0$, and if $\Lambda'(x)=0$ then
\[
\Lambda''(x)=-\frac{u_{xx}(T,x,\Lambda(x))
}{u_{\lambda}(T,x,\Lambda(x))}.
\]

The following result is an adaptation of \cite[Theorem 1.3]{BFG2}. 
It states that if the number of closed solutions for a certain value of $\lambda$, e.g. $\lambda_2$, is known
and the graph of $\Lambda$ has only minima, then the number of closed solutions 
cannot increase for lower values of the parameter, except maybe for two closed solutions corresponding to a bifurcation of the origin and a bifurcation of infinity. A similar
conclusion holds when the graph of  $\Lambda$ has only maxima.

\begin{theo}\label{theo:bifrotated}
Assume that Abel equation \eqref{eq:Abelparametric} satisfies $F_\lambda(t,x,\lambda)>0$ for every $\lambda \in (\lambda_1,\lambda_2)$, $t\in(0,T)$ and $x>0$, and that $u_{xx}(T,\tilde x ,\lambda)<0$ ($u_{xx}(T,\tilde x,\lambda)>0$), for every positive singular closed solution $u(t,\tilde x,\lambda)$ with $\lambda\in[\lambda_1,\lambda_2]$. 
Then 
\[
N(\lambda)\le N(\lambda_2) + 2 \ (N(\lambda)\le N(\lambda_1) + 2)\quad \text{for every }  \lambda \in(\lambda_1,\lambda_2 ).
\]
Moreover, the two possible additional closed solutions correspond to 
a Hopf bifurcation of the origin or a Hopf bifurcation of infinity.

\end{theo}
\begin{proof}
In each of the intervals of the domain of definition of $\Lambda$ there is at most one extremum point, which is a minimum, since otherwise there are two consecutive zeros $x_1<x_2$ of $\Lambda'$ which satisfy $\Lambda''(x_1) \Lambda''(x_2)\leq0$ in contradiction  with the hypothesis. Hence, $\Lambda$ is monotonic or it has a unique  minimum, being alternate monotonic in the latter case, since two consecutive hyperbolic closed solutions have opposite stability.

\medskip
{\em Claim 1.} If $0<x_1<x_2$ then $u(t,x_1,\Lambda(x_1))<u(t,x_2,\Lambda(x_2))$.

If $\Lambda(x_1)=\Lambda(x_2)$, the conclusion follows by the uniqueness of solutions of the initial value problem. If  $\Lambda(x_1)<\Lambda(x_2)$ (resp.~$\Lambda(x_1)>\Lambda(x_2)$) then $u(t,x_1,\Lambda(x_1))$ is a lower (resp.~upper) solution of $x'=F(t,x,\Lambda(x_2))$. The conclusion holds since closed solutions cannot cross lower or upper closed solutions.

\medskip 
{\em Claim 2.} If $\Lambda(x)$ is defined in the interval $(\bar{x}, \tilde{x}]$ with $\bar{x}>0$ then it is also defined at $\bar{x}$.

By claim 1, $x\to u(t,x,\Lambda(x))$ is strictly increasing. Also, there exists $\epsilon>0$ such that $\Lambda$ is 
monotonous continuous in $(\bar x,\bar x+\epsilon)$, so that, 
denoting $\bar{\lambda}= \lim_{x\to \bar{x}}\Lambda(x)$,
\[
u(t,\bar x,\bar \lambda)=\lim_{x\to \bar{x}} u(t,x,\Lambda(x)).
\]
Note that the limit exists since $x\to u(t,x,\Lambda(x))$ is an
increasing function bounded below by $0$. Moreover, $u(t,\bar x,\bar
\lambda)$ is periodic in $t$ since the functions $t\to u(t,x,\Lambda(x))$ are. Thus
$\Lambda(\bar x)=\bar \lambda$.

\medskip 
{\em Claim 3.} If $\Lambda$ is defined in a set $[x_1,x_2) \cup \{x_3\}$, where $x_2<x_3$, then it is also defined at $x_2$.

 By claim 1, $u(t,x,\Lambda(x))<u(t,x_3,\Lambda(x_3))$ for all $x<x_3$, so that if we 
denote $\lambda_2= \lim_{x\to x_2}\Lambda(x)$ then 
\[
u(t,x_2,\lambda_2)=\lim_{x\to x_2} u(t,x,\Lambda(x)),
\]
and we conclude analogously.

\medskip 

As a consequence of claims 1, 2, and 3, if  $D$ is the domain of definition of $\Lambda$ then 
(finiteness of the intervals will be proved below)
$$
D= (0,x_1] \cup [x_2,x_3] \cup \cdots \cup [x_{n-1},x_{n}],
$$
or
$$
D= (0,x_1] \cup [x_2,x_3] \cup \cdots \cup [x_{n},x_{n+1}),
$$
or 	
$$
D=  [x_1,x_2] \cup [x_3,x_4]\cup \cdots \cup [x_{n-1},x_{n}],
$$
or
$$
D=  [x_1,x_2] \cup [x_3,x_4]\cup \cdots \cup [x_{n},x_{n+1}),
$$
where $x_{n+1}\leq \infty$,  $\Lambda(x_i) \in\{ \lambda_1, \lambda_2\}$, for $i=1,\ldots,n$,
and if $x_{n+1}<\infty$ and 
\[
\bar \lambda = \lim_{x\to x_{n+1}} \Lambda(x) <+\infty
\]
then the solution $u(t,x_{n+1},\bar \lambda)$ is unbounded.

As $\Lambda$ has only local minima,  for every $\lambda\in[\lambda_1,\lambda_2]$  the number of 
solutions of $\Lambda(x) = \lambda$ with $x\in (x_i,x_{i+1})$ is less than or 
equal to the number of solutions of $\Lambda(x) = \lambda_2$ with $x\in (x_i,x_{i+1})$, while 
the number of 
solutions of $\Lambda(x) = \lambda$ with $x\in (0,x_1)$ or $x\in (x_n,x_{n+1})$ is less than or 
equal to the number of solutions of $\Lambda(x) = \lambda_2$  with $x\in (0,x_1)$ or $x\in (x_n,x_{n+1})$
plus one, where the extra solution corresponds to a Hopf bifurcation of the origin or infinity, respectively.
In particular, since $N(\lambda_2) < \infty$, we have that $D$ cannot be the union of infinitely many intervals.

\end{proof}	

By the change $\lambda\to-\lambda$, a similar result holds if $F_\lambda(t,x,\lambda)<0$.

Now we shall prove Theorem~\ref{theo:bif}, which determines the stability of 
the singular positive closed solutions. 
We divide the proof into various propositions.
In the following, we assume that $\tilde u(t):=u(t,\tilde x)$ is a singular positive closed solution, and that $(C_1),(C_2)$, and $(C_3)$ hold.

\begin{prop}[\cite{BFG2}]\label{prop:bif0}
For any $\alpha,\beta\in\mathbb{R}$,
\begin{equation}\label{exp}
\sgn\left( u_{xx}(T,\tilde x)\right)= \sgn\left(\int_{0}^{T} F(t,\alpha)G(t,\beta)\,dt \right),
\end{equation}
where
\begin{align*}
F(t,\alpha)&:=(2-\alpha)B (t)+ 2(3-\alpha) A (t) \tilde u(t),\\
G(t,\beta)&:= u_x(t,\tilde x) - \beta \tilde u(t).
\end{align*}
\end{prop}

With equation~\eqref{exp} in mind, the idea for proving Theorem~\ref{theo:bif} is to choose
$\alpha$ and $\beta$ such that the  changes of sign of the corresponding $F(t,\alpha)$ and $G(t,\beta)$ coincide, and consequently $F(t,\alpha)G(t,\beta)$ does not change sign. 

The first step is to determine the changes of sign of $F(t,\alpha)$ and $G(t,\beta)$, which is done in the following two propositions, where we have taken into account Remark \ref{rem:signos}.

By $(C_2)$, $2A(t)\tilde u(t)+B(t)$ has at most one simple zero in each of the intervals $[0,t_A]$, $[t_A,T]$. 
In view of  the signs of $A$ and $B$, it can be proved that there are at least two simple zeros,  denoted  by $t_1$ and $t_2$, such
that
\[0<t_1<t_{B_1}<t_A<t_2<t_{B_2}\leq T.\]
Moreover, regarding $(C_1)$ and Remark \ref{rem:signos}, $2A(t)\tilde u(t)+B(t)$ is positive in $[0,t_1)\cup(t_2,T]$ and negative in $(t_1,t_2)$.

\begin{prop}\label{prop:alpha}
for all $\alpha \in \mathbb{R}$, $F(t,\alpha)$ has at most two changes of sign in $(0,T)$.
More precisely,
\begin{enumerate}
 \item $F(t,\alpha)=0$ is the graph of a smooth function $\alpha(t)$ defined for
  every $t\neq t_1,t_2$,
 \item $F(t,\alpha) > 0$ for $\alpha < \alpha(t),\ t \in [0,t_1) \cup (t_2,T]$, and for $\alpha > \alpha(t),\ t \in [t_1,t_2]$,
 \item  $F(t,\alpha) < 0$ for $\alpha > \alpha(t),\ t \in [0,t_1) \cup (t_2,T]$, and for $\alpha < \alpha(t),\ t \in [t_1,t_2]$,
\item $\alpha(t)$ is strictly decreasing in its domain of definition,
 \item for every $i=1,2$,
 \[
 \lim_{t\to t_i^{\pm}} \alpha(t)=\pm\infty,
 \]
 \item $\alpha(T)\geq \alpha(0)=2$.
 \end{enumerate}
\end{prop}

\begin{proof}

First, we observe that $F(t,\alpha)=0$ if and only if $\alpha=\alpha(t)$ where
\[ \alpha(t)= 2+\dfrac{2A(t)\tilde u(t)}{2A(t)\tilde u(t)+B(t)}.\]
Then, by $(C_2)$, the function $\alpha(t)$ is smooth, so that $(1)$ holds. Moreover, since $2A(t)\tilde u(t)+B(t)$ 
changes sign at $t_1,t_2$, and $F(t_A,0)=2B(t_A)<0$ then $(2)$ readily follows (see Fig.~\ref{fig:1}).

To prove $(4)$, it suffices to observe that $\sgn (\alpha'(t))=\sgn\left(v(t,\tilde u(t))\right)
$ and that $\sgn\left(v(t,\tilde u(t))\right) < 0$ by $(C_3)$. Moreover, as $A(t_i)\neq 0,\ i = 1,2$, we have $(5)$. 

By $(C_1)$ and Remark \ref{rem:signos}, $A(T)$ and $B(T)$ are non-negative and not simultaneously zero, so that
\[
\dfrac{2A(T)u(T,x)}{2A(T)u(T,x)+B(T)}\geq 0.
\]
Hence, $\alpha(T)\geq 2=\alpha(0)$, and $(6)$ follows.

Finally, as $F(t,\alpha)=0$ is the graph of $\alpha(t)$,
then $(1)-(6)$ imply that, for every fixed $\alpha$, the function $t\to F(t,\alpha)$ has at most two changes of sign in $(0,T)$.

\begin{figure}[h] 
\begin{center}
\begin{tikzpicture}[scale=1]
\clip (-4.5,-2) rectangle (4.5,2);
\draw[-] (-3.5,0) -- (3.5,0);
\draw[-] (-3.5,-2) -- (-3.5,2);
\draw[-] (3.5,-2) -- (3.5,2);
\draw[dashed] (-1,-2) -- (-1,2);
\draw[dashed] (1,-2) -- (1,2);
\draw[dashed] (-3.5,.19) -- (3.5,.19);
\draw[dashed] (-3.5,.81) -- (3.5,.81);
\filldraw (-1,0) circle (2pt);\node[fill=white] at (-1-0.10,-0.55) {$t_1$};
\filldraw ( 1,0) circle (2pt);\node[fill=white] at ( 1+0.10,-0.55) {$t_2$};
\filldraw ( 0.2,0) circle (2pt);\node[fill=white] at ( 0+0.10,-0.55) {$t_A$};
\node[fill=white] at ( -2.3+0.10,1.55) {${\scriptstyle {F(t,\alpha)<0}}$};
\node[fill=white] at ( 2.3+0.10,-1.55) {${\scriptstyle {F(t,\alpha)>0}}$};
\node[fill=white] at ( -2.6+0.10,-1.55) {${\scriptstyle {F(t,\alpha)>0}}$}; 
\draw[smooth=0.7,ultra thick,oro!50!rojo] plot[domain=-3.5:-1.05,samples=200] ({\x},{0.5+\x/(\x*\x-1)});
\draw[smooth=0.7,ultra thick,oro!50!rojo] plot[domain=-0.95:0.95,samples=200] ({\x},{0.5+\x/(\x*\x-1)});
\draw[smooth=0.7,ultra thick,oro!50!rojo] plot[domain=1.05:3.5,samples=200] ({\x},{0.5+\x/(\x*\x-1)});
\node at (2.5,1.5) {$\alpha(t)$};
\end{tikzpicture}
\end{center}
\caption{Sketch of $\alpha(t)$.}\label{fig:1}
\end{figure}
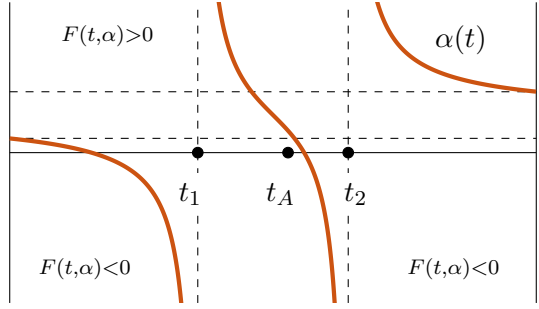

\end{proof}

Note that  $F(t,\alpha)=0$ always defines
the graph of a function, so that hypotheses of Theorem~\ref{theo:bif} 
are imposed to determine its properties. In 
particular, $(C_2)$ implies that it has two asymptotes,
and $(C_3)$ implies the monotonicity of $\alpha(t)$. 

A similar result holds for the zeros of $G(t,\beta)$, which
are determined by the zeros of a given function whose number of extrema and their nature are determined by $(C_2)$.

\begin{prop}\label{prop:beta}
There exist $\beta_0,\beta_1,\beta_2$ such that $G(t,\beta)$ has two changes of sign in $(0,T)$ for
every $\beta\in(\beta_1,\beta_2)$, $\beta\neq \beta_0$, and no zeros for
$\beta\not\in[\beta_1,\beta_2]$. More precisely,
\begin{enumerate}
\item $G(t,\beta)=0$ is the graph of a positive closed smooth function $\beta(t)$ defined for $t\in[0,T]$,
\item $G(t,\beta)>0$ for $\beta<\beta(t)$ and $G(t,\beta)<0$ for $\beta>\beta(t)$,
\item $\beta(t)$ has exactly two extrema: a maximum at $t_1$ and a minimum at $t_2$.
\end{enumerate}
\end{prop}

\begin{proof}
First, we notice that $G(t,\beta) = 0$ if and only if $\beta = \beta(t)$, where 
\begin{equation}\label{eq:beta}
\beta(t)=\frac{ u_x(t,\tilde x)}{\tilde u(t)}.
\end{equation}
As $\tilde u(t) > 0$, we have that $\beta(t)$ is defined for all $t\in\mathbb{R}$. Furthermore, deriving in \eqref{eq:Abel} with respect to $x$ and using that $ u_x(0,x)=1$, we obtain that \[ u_x(t,\tilde x) = \operatorname{exp}\left(\int_0^t \left(3 A(t) \tilde u^2(t) + 2 B(t) \tilde u(t)\right) dt\right),\] and therefore $ u_x(t,\tilde x) > 0$. Thus, $\beta(t) > 0$. Since $\tilde u(t)$ is singular, both $\tilde u(t)$ and $ u_x(t,\tilde x)$ are closed, and then $\beta(0)=\beta(T)=:\beta_0$. Hence we conclude $(1)$. 

Since $G_{\beta}(t,\beta)=-u(t,x)<0$, then $(2)$ follows. 

To prove $(3)$, we first note that
\begin{equation}\label{ecu:beta_prima}
\beta'(t)=\left(2A(t)\tilde u^2(t)+B(t)\tilde u(t) \right)\beta(t).
\end{equation}
Since $\beta(t) > 0$, we conclude that $\beta'(t) = 0$ if and only if $t = t_1$ or $t = t_2$. Moreover, as $2A(t)\tilde u^2(t)+B(t)\tilde u(t) > 0$ for $t \in [0,t_1) \cup (t_2,T]$, we have that $\beta(t)$ has a maximum at $t_1$ and a minimum at $t_2$.

Finally, writing $\beta_1 = \beta(t_1)$ and $\beta_2 = \beta(t_2)$, we obtain that $G(t,\beta)$ has two changes of sign in $(0,T)$ for every $\beta\in(\beta_1,\beta_2),\ \beta\neq \beta_0$, and no zeros for $\beta\not\in[\beta_1,\beta_2]$. 

\begin{figure}[h] 
\begin{center}
\begin{tikzpicture}[scale=1]
\clip (-4.5,-1) rectangle (4.5,2);
\draw[-] (-3.5,0) -- (3.5,0);
\draw[-] (-3.5,-1) -- (-3.5,2);
\draw[-] (3.5,-1) -- (3.5,2);
\draw[dashed] (-1,-1) -- (-1,2);
\draw[dashed] (1,-1) -- (1,2);
\draw[dashed] (-3.5,1) -- (3.5,1);
\filldraw (-1,0) circle (2pt);\node[fill=white] at (-1-0.10,-0.55) {$t_1$};
\filldraw ( 1,0) circle (2pt);\node[fill=white] at ( 1+0.10,-0.55) {$t_2$};
\draw[smooth=0.7,ultra thick,oro!50!rojo] plot[domain=-3.5:-1,samples=200] ({\x},{-2*(\x+1)*(\x+1)/2.5/2.5/3+2/3+1});
\draw[smooth=0.7,ultra thick,oro!50!rojo] plot[domain=-1:1,samples=200] ({\x},{(\x*\x*\x/3-\x)+1});
\draw[smooth=0.7,ultra thick,oro!50!rojo] plot[domain=1:3.5,samples=200] ({\x},{2*(\x-1)*(\x-1)/2.5/2.5/3-2/3+1});
\node at (0.2,1.5) {$\beta(t)$};
\node at (-3.5-0.4,1) {$\beta_0$};
\end{tikzpicture}
\end{center}
\caption{Sketch of $\beta(t)$.}\label{fig:beta}
\end{figure}
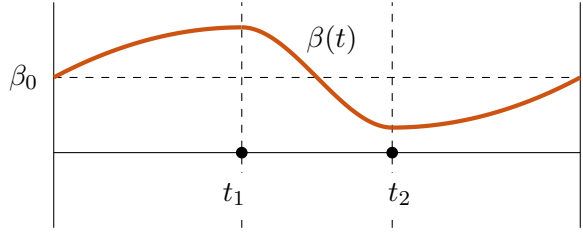

\end{proof}

Recall that, by Proposition~\ref{prop:bif0},
\[\sgn( u_{xx}(T,\tilde x))= \sgn\left(\int_{0}^{T} F(t,\alpha)G(t,\beta)\,dt \right).\]
Therefore, in order to complete the proof of Theorem~\ref{theo:bif}, it only remains to prove that there exist $\alpha,\beta$ such that the 
changes of sign of $F(t,\alpha)$ and $G(t,\beta)$ coincide. 

\begin{prop}\label{prop:bif3}
There exist $\alpha,\beta$ such that the changes of sign of $F(t,\alpha)$ and $G(t,\beta)$ coincide. Moreover, for these $\alpha$ and $\beta$,
\[ \int_0^T F(t,\alpha)G(t,\beta)\,dt  < 0.\]
\end{prop}

\begin{proof}

Let $\beta(t)$ be the closed smooth function defined
by~\eqref{eq:beta} and write $\beta_0 = \beta(0)$.  By
Proposition~\ref{prop:beta}, $\beta(t)$ has a maximum at $t_1$ and a
minimum at $t_2$. Set $\beta_1 = \beta(t_1)$ and $\beta_2 =
\beta(t_2)$, and let $t_0$ be the unique value in $(0,T)$ with
$\beta(t_0) = \beta_0$.

We distinguish three cases in accordance with the relative
position of $\alpha(0)=2$, $\alpha(t_0)$, and $\alpha(T) \geq 2$.

\begin{itemize}
\item If $\alpha(t_0) \in (\alpha(0),\alpha(T))$, then take $\alpha =\alpha(t_0)$, $\beta=\beta_0$. Then $F(t,\alpha)$ and $G(t,\beta)$ only change sign in $(0,T)$ at $t_0$. 
By Proposition~\ref{prop:alpha}, $F(t,\alpha) < 0$ for $t \in [0,t_0)$ and $F(t,\alpha) > 0$ for $t \in (t_0,T]$. Moreover,
by Proposition~\ref{prop:beta}, $G(t,\alpha) > 0$ for $t \in [0,t_0)$ and $G(t,\alpha) < 0$ for $t \in (t_0,T]$. Consequently, $F(t,\alpha)G(t,\beta) < 0$
for all $t \in [0,t_0) \cup (t_0,T]$, and the result holds. 
 
\item $\alpha(t_0)<\alpha(0)$. Since $\beta(t)$ has a maximum at $t_1$, and
 is strictly monotonic for $t\in(0,t_1) \cup (t_1, t_0)$, there exist two continuous monotonic functions
$T_1,T_2$, defined in $(\beta_0,\beta_1)$ such that $T_1(\beta(t))=t,\ t \in (0,t_1)$ and $T_2(\beta(t))=t,\ t \in (t_1,t_0)$. Notice that 
$0 < T_1(\beta)<t_1<T_2(\beta)<t_0,\ \beta \in (\beta_0, \beta_1)$, and that $T_1, T_2$ have opposite monotonicity. 

Now, let us define the
continuous function
\[
d(\beta)=\alpha(T_1(\beta))-\alpha(T_2(\beta)),\quad \beta\in (\beta_0,\beta_1).
\]
Since $\alpha(t)\to\pm\infty$ as $t\to t_1^{\pm}$, then $\lim_{\beta\to\beta_1}d(\beta)=-\infty$. 
On the other hand, $d(\beta_0)=\alpha(0)-\alpha(t_0)>0$. By continuity, there exists $\bar \beta$ such that $d(\bar \beta)=0$. For $\alpha=\alpha(T_1(\bar \beta))$ and $\beta=\bar \beta$, $F,G$
have the same changes of sign: exactly two and both in $(0,t_0)$. 

On the one hand, $\alpha=\alpha(T_1(\bar \beta))<\alpha(0)=2$, so that $F(t,\alpha)>0$ for $t$ close to zero by Proposition \ref{prop:alpha}. On the other hand, for $t$ close to zero, $\beta(t)<\beta=\bar \beta$, so that $G(t,\alpha)<0$  by Proposition \ref{prop:beta}. Consequently, $F(t,\alpha)G(t,\beta)\leq 0$, and
\[ \int_0^T F(t,\alpha)G(t,\beta)\,dt <0.\]

\item $\alpha(t_0)>\alpha(T)$. Since $\beta(t)$ has a minimum at $t_2$, and
 is strictly monotonic for $t \in(t_0,t_2) \cup (t_2,T)$, then there exist two continuous monotonic functions $T_1,T_2$, defined in $(\beta_2,\beta_0)$,
such that $T_1(\beta(t))=t,\ t \in (t_0,t_2)$ and $T_2(\beta(t))=t,\ t \in (t_2,T)$. Notice that
$t_0 < T_1(\beta)<t_2<T_2(\beta)<T,\ \beta \in (\beta_2, \beta_0)$, and that $T_1,T_2$ have opposite monotonicity. Now, if $d(\beta) = \alpha(T_1(\beta))-\alpha(T_2(\beta)),\ \beta \in (\beta_2,\beta_0)$, then $\lim_{\beta\to\beta_2} d(\beta)=-\infty$ and $d(\beta_0)=\alpha(t_0)-\alpha(T)>0$, and we conclude as in the previous case.
 \end{itemize}
\end{proof}

\section{Sufficient Criteria}

The following results establish sufficient conditions for $(C_2)$ and $(C_3)$ to be satisfied without assuming knowledge of the positive singular closed solutions of \eqref{eq:Abel}. 

The first result is an adaptation of Proposition 5 of \cite{BFG2} to the case $\gamma=1$. In order to obtain $(C_2)$, we define 
\[
\phi(t)=-B(t)/(2A(t)).
\]
By condition $(C_1)$ and Remark \ref{rem:signos}, $\phi(t)\geq 0$ if and only if $t\in [0,t_{B_1}]\cup [t_{A},t_{B_2}]$.

\begin{figure}[h] 
\begin{center}
\begin{tikzpicture}[scale=3]
\clip (-0.5,-1) rectangle (2.5,1);
\draw[-] (0,0) -- (2,0);
\draw[-] (0,-1) -- (0,1);
\draw[-] (2,-1) -- (2,1);
\draw[dashed] (.3,-1) -- (.3,1); \node[fill=white] at (.3+0.05,-0.2) {\small $t_{B_1}$};
\draw[dashed] (.8,-1) -- (.8,1); \node[fill=white] at (0.8+0.05,-0.2) {\small $t_{A}$};
\draw[dashed] (1.4,-1) -- (1.4,1); \node[fill=white] at (1.4+0.05,-0.2) {\small $t_{B_2}$};
\draw[smooth=0.7,ultra thick,oro!50!rojo] plot[domain=0.01:0.79,samples=200] ({\x},{-((\x-1.4)*(\x-0.3))/(20*(\x-0.8)*\x)});
\draw[smooth=0.7,ultra thick,oro!50!rojo] plot[domain=0.81:2,samples=200] ({\x},{-((\x-1.4)*(\x-0.3))/(20*(\x-0.8)*\x)});
\node at (1.75,0.75) {\small $\phi(t)$};
\filldraw (.3,0) circle (1pt);
\filldraw (.8,0) circle (1pt);
\filldraw (1.4,0) circle (1pt);
\end{tikzpicture}
\end{center}
\caption{Sketch of $\phi(t)$.}\label{fig:2}
\end{figure}
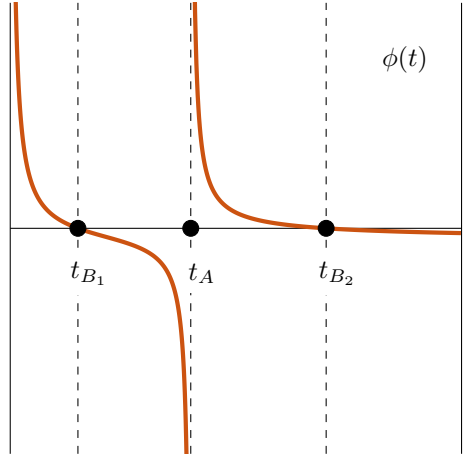

Now, we impose some sufficient conditions in order that any positive bounded solution $u(t,x)$ crosses the graph of $\phi$ in at most two points. These conditions are quite restrictive, but can be verified computationally.

\begin{prop}[\cite{BFG2}]\label{prop:upper}
Let $u(t,x)$ be a positive singular closed solution of \eqref{eq:Abel} and suppose that $(C_1)$ holds.
	Let $J_1=(0,t_{B_1})$ and $J_2=(t_{A},t_{B_2})$. If the function $P$, defined in \eqref{eq:upper}, has at most one zero in each $J_i,\ i = 1,2$, then $u(t,x)-\phi(t)$ has  a unique simple zero in each $J_i, i=1,2${, i.e.,} condition $(C_2)$ holds.
\end{prop}

\begin{proof}
Firstly, we observe that
\begin{align*}
\phi' - B \phi^2 - A  \phi^3 & = \frac{B A '-B 'A }{2A ^2} - \frac{ B ^3}{4A ^2} + 	\frac{ B ^3}{8A ^2} \\  & = \frac{4(B A '-B 'A ) - B ^3 }{8A ^2} = \frac{P}{8 A^2}.
\end{align*}
If $\phi' - B \phi^2 - A  \phi^3$ has no zeros in $J_i$, then $\phi$ is an upper or lower solution of \eqref{eq:Abel} and therefore the graphs of $u(t,x)$ and $\phi$ coincide in at most one point.

If $\phi' - B \phi^2 - A  \phi^3$ has one zero in $J_i$, then $ \phi$ changes from an upper (resp.~lower) solution to a lower (resp.~upper) solution of \eqref{eq:Abel} in that interval. In any case, since $u(t,x)$ is bounded and the graph of $\phi$ in $J_i$ goes from zero to infinity, $u(t,x)$ intersects $\phi$ at one point in $J_i$.

Therefore,  $2 A(t) u(t,x) + B(t)$ has at most one zero in $(0,t_{B_1})$ and at most one zero in $(t_A,t_{B_2})$.

Finally, since \[2 A(0) u(0,x) + B(0) = B(0) > 0,\] \[2 A(t_{B_1}) u(t_{B_1},x) + B(t_{B_1}) = A(t_{B_1}) u(t_{B_1},x) < 0,\] we have that $2 A(t) u(t,x) + B(t)$ has at least one zero in $(0,t_{B_1}),$ and since 
\[2 A(t_A) u(t_A,x) + B(t_A) = B(t_A) < 0,\] \[2 A(t_{B_2}) u(t_{B_2},x) + B(t_{B_2}) = A(t_{B_2}) u(t_{B_2},x) > 0,\] we have that $2 A(t) u(t,x) + B(t)$ has at least one zero in $(t_A,t_{B_2}).$
\end{proof}

We now obtain a sufficient condition for $(C_3)$ to hold, which can be
computed.  Note that $(C_3)$ is equivalent to imposing that the graph
of every singular positive closed solution $\tilde u$ is
contained in the region $v(t,x)< 0$, where $v(t,x)$ is the function
defined in \eqref{ecu:v}. In order to control the intersections of the
solutions with the complementary region $v(t,x)\geq 0$, consider the
derivative of the solutions with respect to the vector field
\eqref{eq:Abel}, i.e.,
\[
\dot v(t,x)= v_t(t,x)+v_x(t,x)(A(t)x^3+B(t)x^2). 
\]
Controlling the common zeros of $v$ and $\dot v$, we obtain a sufficient condition for $(C_3)$ to hold.

Let
\[
\begin{split}
v^{-1}(0)&=\{(t,x): 0<t<T, x>0, v(t,x)=0 \},\\
\dot v^{-1}(0)&=  \{(t,x): 0<t<T, x>0, \dot v(t,x)=0 \},
\end{split}\]
and denote $S = [0,T]\times [0,\infty)$.
 
\begin{prop}\label{PropC3}
If $v(t,0)=A'(t)B(t)-A(t)B'(t)<0$ for all $t\in[0,T]$, $v(0,x)<0$ and $v(T,x)<0$ for all $x\geq 0$, and $v^{-1}(0) \cap \dot v^{-1}(0)=\emptyset$, then condition $(C_3)$ holds.	
\end{prop}
\begin{proof}
	 $v^{-1}(0) \cap \dot v^{-1}(0)=\emptyset$, the set $v^{-1}(0)$ has no 
	singular points, so it consists of regular curves. By {the one-point} compactification
		of the region $[0,T]\times [0,+\infty)$, we may assume that they are closed, so that, by the Jordan curve theorem, each of these regular curves
	divides the space into two regions.
	Since $v(t,0),v(0,x),v(T,x)<0$ for all $t\in[0,T]$ and $x\geq 0$, then there
		is a connected region $W$ in $v(t,x)<0$ containing the points of the form
		${(t,0)},(0,x),(T,x)$, for all $t\in[0,T]$ and $x\geq 0$.

	From the hypothesis, we have that
	$$
	\dot v(t,x)=\langle (v_t(t,x),v_x(t,x)), (1,A(t)x^3+B(t)x^2) \rangle
	$$
	has definite sign on $v^{-1}(0)$,
	where $\langle\cdot.\cdot\rangle$ is the ordinary scalar product in $\mathbb{R}^2$ and $(1,A(t)x^3+B(t)x^2)$ is the vector field defined by \eqref{eq:Abel}. By the Jordan curve theorem, we can fix an orientation for any given regular curve contained in the set $v^{-1}(0)$, and the field has either that same orientation at each point of the curve or the opposite orientation at each point of the curve. Hence, one of the regions into which the curve divides the space is positively invariant and the other negatively invariant.

	In any case, since any bounded solution $u(t,x)$ of \eqref{eq:Abel} starts and ends in the connected region $W$ 
	the graph of $u(t,x)$  does not intersect $v^{-1}(0)$, and consequently condition $(C_3)$ holds.	
\end{proof}	

The following result provides a simple sufficient condition that implies $v^{-1}(0) \cap \dot v^{-1}(0)=\emptyset$, which will be used in the examples. 

\begin{coro}\label{CorC3}
Let
	\begin{equation}\label{eq:elim}
		Q(t)=B(t) (A(t) B''(t)-B(t)A''(t)) + 3 B'(t)(B(t)A'(t) - A(t)B'(t)).
	\end{equation}
If $Q(t)$ has no zeros in $(0,T)$ or $v(\bar t,x)=0$ does not have positive solutions for each zero $\bar t$ of $Q(t)$ in $(0,T)$
then $v^{-1}(0) \cap \dot v^{-1}(0)=\emptyset$.
\end{coro}

\begin{proof}
Observe that 
\[Q(t)=4 \left(2 A(t) B(t) x^2+B^2(t) x+3 B'(t)\right) v(t,x)-4 B(t)\, \dot v(t,x).\] Thus, if $v(t,x) = \dot v(t,x) = 0$ then $Q(t) = 0$. 
So, $v^{-1}(0) \cap \dot v^{-1}(0)=\emptyset$, since otherwise there exist $0<\bar t<T$ and $\bar x>0$ such that $v(\bar t,\bar x)=\dot v(\bar t,\bar x)=Q(\bar t)=0$, in contradiction with the hypothesis. 
\end{proof}

\section{Example of Application}\label{Sect4}

In this section, we prove Theorem~\ref{theo:example}, i.e., that \eqref{eq:Abel} has at most two positive closed solutions
when $T=1$ and \[A(t) = t(t-t_A)\quad \text{and}\quad B(t) = (t-t_B)(t-1),\quad t_A,t_B\in\mathbb{R}.\]

In either of the following cases, the known methods allow it to be proved that
\eqref{eq:Abelquadquad} has at most one simple positive closed solution:

\begin{enumerate}
	\item\label{Pl} $t_A\not\in(0,1)$ or $t_{B}\not\in(0,1)$.
	\item\label{AGG} $t_A\in(0,1), t_{B}\in(0,1)$, and {$t_A \in (0,t_{B}]$}. 
\end{enumerate}

In case \eqref{Pl}, either $A$ or $B$ has no zeros in $(0,1)$:
if $A$ has no zeros, it was proved in \cite{Pliss} that \eqref{eq:Abelquadquad} has at most one positive closed solution, 
while if $B$ has no zeros, the proof was given in \cite{GL}. In case \eqref{AGG},  it was proved in
\cite{AGG} that if for some $\alpha,\beta\in\mathbb{R}$ the function $\alpha A+\beta B$ does not
vanish identically and does not change sign in $(0,1)$ then the Abel equation has at most  one
positive closed solution. Hence, if we consider a linear combination of the form $\alpha A(t)+B(t)$,
its discriminant $d(\alpha)$ is a degree-two polynomial in $\alpha$ with leading coefficient $t_A^2$. Therefore, there exists
$\alpha$ such that $d(\alpha)\leq0$ (and so $\alpha A(t)+B(t)$ does not change  sign) if and only if its discriminant 
is greater than or equal to zero. But this discriminant is 
\[
\disc(d)=-(1-t_A)t_B(t_A-t_B),
\]
which is non-negative if and only if $t_A \in (0,t_B]$, and the result follows.

Hence, to prove  Theorem~\ref{theo:example} we may assume that  $0 < t_B < t_A < 1$.  
We shall divide the proof into two parts, first proving that \eqref{eq:Abel} satisfies the hypotheses of Theorem~\ref{theo:bif}, and 
then using Theorem~\ref{theo:bifrotated} to show 
that there are at most two positive closed solutions.

\subsection{Semistability of the singular solutions} As it is immediate to check that $(C_1)$ holds,  it only remains to verify that $(C_2)$ and $(C_3)$ hold to apply Theorem~\ref{theo:bif}.

Let us see that $(C_2)$ holds. By Proposition \ref{prop:upper}, it suffices to prove
that $P(t)=4(B (t)A '(t)-B '(t)A (t)) - B ^3(t)$ has at most one zero
in each of the intervals $J_1=(0,t_{B})$, $J_2=(t_{A},1)$. First we need the following lemma.

\begin{lema}\label{lema3}
The polynomial $A'(t)B(t)-A(t)B'(t)$ is negative for every $t$.
\end{lema}

\begin{proof}
$A'(t)B(t)-A(t)B'(t)$ is a quadratic polynomial in $t$ with coefficients in $\mathbb{R}[t_A,t_B]$ whose discriminant and leading coefficient are $-(1-t_A)t_B(t_A-t_B) < 0$ and $t_A-t_B-1 < 0$, respectively.
\end{proof}

\begin{prop}\label{prop1}
The function $P(t)$ is negative in  $(0,t_B) \cup (t_A,1)$. In particular, condition $(C_2)$ holds.
\end{prop}
 
\begin{proof}
By hypothesis, $B(t) > 0$ for all $t \in (0,t_B)$, and, by Lemma \ref{lema3}, we have that $A'(t)B(t)-A(t)B'(t) < 0$ for all $t$. Thus, \[P(t) = 4(A'(t)B(t)-A(t)B'(t)) - B^3(t) < 0\] for all $t \in (0,t_B)$. 

Let us assume that $t_A < t < 1$. Observe that 
\[P(t) = -(t^3(t-t_A)^3 + (1-t)(2t-t_A)(t-t_B) + t(t-t_A)).\] 
Since $t^3(t-t_A)^3 > 0, (1-t)(2t-t_A)(t-t_B) > 0$, and $t(t-t_A) > 0$ for all $t \in (t_A,1)$, we conclude that $P(t)$ is negative for all $t \in (t_A,1)$.

Finally, since $P(t)$ has no zeros in $(0,t_B) \cup (t_A, 1)$, condition $(C_2)$ holds by Proposition \ref{prop:upper}.
\end{proof}

Now let us prove that $(C_3)$ is fulfilled by using Proposition~\ref{PropC3} and Corollary~\ref{CorC3}.
Recall that $v(t,x)=B(t)(2A(t)x+B(t))^2+P(t)$.
By Lemma~\ref{lema3}, we have that $v(0,x)<0$ and $v(1,x)<0$ for all $x\geq 0$, and $v(t,0)<0$ for all $t\in[0,1]$. 
So it suffices to show that $v(\bar{t},x)=0$ has no positive solution for each zero $\bar{t}$ of $Q(t)$ in the interval $(0,1)$.

In our setting, the function $Q(t)$ in \eqref{eq:elim} is the following cubic polynomial in $t$ with coefficients in $\mathbb{R}[t_A,t_B]$:
\begin{align*}
Q(t) = & -4(1-t_A+t_B)\, t^3+ ((t_B^2+12t_B+1)-(1+t_B)t_A)\, t^2\\ & -2t_B (t_B+4t_A+1)\, t+ t_B(3t_A(t_B+1)-2t_B).
\end{align*}
We claim that $Q$ has exactly one zero in $(0,1)$. To prove this, we shall apply Sturm's theorem (\cite[Theorem 2.50]{ARAG}), but first we need to introduce some additional notation.

Let $\mathcal{S}$ be the following four-term sequence:
\[\mathcal{S}_0 = Q(t),\]
\[\mathcal{S}_1 = Q'(t),\]
\[\mathcal{S}_2 = -\operatorname{Rem}(\mathcal{S}_0,\mathcal{S}_1),\]
\[\mathcal{S}_3 = -\operatorname{Rem}(\mathcal{S}_1,\mathcal{S}_2),\]
where $\operatorname{Rem}(\mathcal{S}_i,\mathcal{S}_{i+1})$ is the remainder of dividing $\mathcal{S}_i$ by $\mathcal{S}_{i+1}$ as polynomials in $t$.
The sequence $\mathcal{S}$ is the so-called signed remainder sequence of $Q(t)$ and $Q'(t)$ (see \cite[Definition 1.7]{ARAG})).

\begin{lema}\label{lema2}
With the above notation, $\mathcal{S}_3 > 0$.
\end{lema}

\begin{proof}
A direct computation shows that $\mathcal{S}_3$ is equal to \[36(1-t_A)(1-t_B)^2 t_B(t_A-t_B)(1-t_A+t_B)\, \frac{f(t_A,t_B)}{g^2(t_A,t_B)}\]
 where
 \begin{align*}
  f(t_A ,t_B) 
  =\ & 3t_B^4-6t_At_B^3+3t_A^2t_B^2+506t_At_B^2-506t_B^2 -\\ & -506t_A^2t_B+506t_At_B+3t_A^2-6t_A+3
 \end{align*} 
  and
  \begin{align*}
   g(t_A ,t_B) =\ & t_B^4 - 2t_At_B^3 + t_A^2t_B^2 - 98t_At_B^2 + 98t_B^2 + \\ & + 98t_A^2t_B  - 98t_At_B + t_A^2 - 2t_A+1 .
   \end{align*}
    Since $0 < t_B < t_A < 1$, we have that the sign of $\mathcal{S}_3$ is the same as the sign of $f(t_A,t_B)$. Now, it suffices to observe that
     \begin{align*}
     	f(t_A,t_B) =\ & (3t_B(t_A-t_B)+506(1-t_A))t_B(t_A-t_B)+3(1-t_A)^2 > 0
     \end{align*}
     	 to get the desired result.
\end{proof}

\begin{prop}\label{prop2}
The cubic polynomial $Q(t)$ has exactly one zero in $(0,1)$. Moreover, this zero lies in $(t_B,1)$.
\end{prop}

\begin{proof}
Let us see that $Q(t)$ has exactly one root in $(t_B,1)$.
On the one hand, the number of sign variations of $\mathcal{S}$ at $t_B$, $\operatorname{Var}(\mathcal{S},t_B)$, is $2$. Indeed, when evaluating $\mathcal{S}$ at $t_B$ one has 
\[\mathcal{S}_0 = Q(t_B) = 3(1-t_B)^2t_B(t_A-t_B) > 0,\]
\[\mathcal{S}_1 = Q'(t_B) = -10(1-t_B)t_B(t_A-t_B) < 0,\]
\[\mathcal{S}_3 > 0,\]
where the last inequality follows by Lemma \ref{lema2}. Notice that
the number of sign variations at $t_B$ is equal to two regardless of
the sign of $\mathcal{S}_2$ at $t_B$. On the other hand, the number of
sign variations of $\mathcal{S}$ at $1$,
$\operatorname{Var}(\mathcal{S},1)$, is $1$. Indeed, when evaluating
$\mathcal{S}$ at $1$ one has
\[\mathcal{S}_0 = Q(1) = -3(1-t_A)(1-t_B)^2 < 0,\]
\[\mathcal{S}_1 = Q'(1) = -10(1-t_A)(1-t_B) < 0,\]
\[\mathcal{S}_3 > 0,\]
where the last inequality follows by Lemma \ref{lema2}. Again notice
that the number of sign variations at $1$ is equal to one regardless of the sign
of $\mathcal{S}_2$ at $1$.

Now, by Sturm's theorem (\cite[Theorem 2.50]{ARAG}) we conclude that the number of real roots of $Q(t)$ in $(t_B,1)$ is equal to 
\[\operatorname{Var}(\mathcal{S},t_B) - \operatorname{Var}(\mathcal{S},1) = 1.\]

Next we prove that $Q(t)$ has no roots in $(0,t_B]$ by using the
Budan-Fourier theorem (\cite[Theorem 2.35]{ARAG}).  Let
$\operatorname{Der}(Q)$ be the list $Q(t), Q'(t), Q''(t),$
$Q'''(t)$. Let us compute the number of sign variations of
$\operatorname{Der}(Q)$ at the borders of the intervals.
\begin{itemize}
\item $\operatorname{Var}(\operatorname{Der}(Q),0)=3$. Indeed, since $0 < t_B < t_A < 1$, we have that 
\[Q(0)=t_B(t_A(3t_B+1)+2(t_A-t_B)) > 0,\]
\[Q'(0) = -2t_B(4t_A+t_B+1) < 0,\]
\[
Q''(0) = (2t_B+2)(1-t_A)+2t_B(t_B+11) > 0,
\]
\[Q'''(0) = -24(1+t_B-t_A) < 0.\]
\item $\operatorname{Var}(\operatorname{Der}(Q),t_B)=3$. Indeed, since $0 < t_B < t_A < 1$ and $t_B<1/2$, we have that
\[Q(t_B) > 0,\]
\[Q'(t_B) < 0,\]
\[
Q''(t_B)  = 22t_B(t_A-t_B)+2(1-t_A) > 0
\]
\[Q'''(t_B) = Q'''(0) < 0.\]
\end{itemize}
Therefore, by the Budan-Fourier theorem, we obtain that the number of
roots of $Q$ in $(0,t_B]$ is less than or equal
to \[\operatorname{Var}(\operatorname{Der}(Q),0) -
\operatorname{Var}(\operatorname{Der}(Q),t_B) = 0,\] i.e.,
$Q(t)$ has no roots in $(0,t_B]$, so we conclude.
\end{proof}

\begin{prop}\label{prop3}
Let $\bar t \in (t_B,1)$ be the unique real root of $Q(t)$ in $(0,1)$.
Then  \[v(\bar t, x) = 0\] has no positive solutions. In particular, $(C_3)$ holds.
\end{prop}
\begin{proof}
First we note that 
\[
v(\bar t, x) = 4 A^2(\bar t)B(\bar t)x^2+4 A(\bar t)B^2(\bar t)x+4 (A'(\bar t)B(\bar t)-A(\bar t)B'(\bar t)).
\]
We distinguish three cases:
\begin{enumerate}
\item If $\bar t \in (t_B,t_A)$ then $A^2(\bar t) B(\bar t) < 0$, $A(\bar t) B^2(\bar t) < 0$, and $A'(\bar t) B(\bar t)-B'(\bar t) A(\bar t) < 0$  by Lemma \ref{lema3}. So all the coefficients in $x$ of $v(\bar t,x)$ are negative, and we conclude that $v(\bar t, x)$ has no positive roots.
\item If $\bar t = t_A$ then $v(\bar t, x) = A'(t_A)B(t_A)\neq0$.
\item If $\bar t \in (t_A,1)$ then $P(\bar t) < 0$  by Proposition \ref{prop1}. Now we only need to observe that the discriminant of $v(\bar t,x)$ is equal to $- A^2(\bar t) B(\bar t) P(\bar t) < 0$ to conclude that $v(\bar t, x)$ has no real roots.
\end{enumerate}

By Lemma~\ref{lema3}, $A'(t)B(t)-A(t)'B(t)<0$ for all $t\in(0,1)$. Also, $v(0,x) = -t_A t_B <0$, $v(1,x) = -(1-t_A)(1-t_B)<0$, and $v^{-1}(0) \cap \dot v^{-1}(0)=\emptyset$ by Corollary \ref{CorC3}. So, by Proposition \ref{PropC3}, we are done.
\end{proof}

Since $(C_1)$, $(C_2)$, $(C_3)$ hold, Theorem~\ref{theo:bif} implies 
that if $u(t,\tilde x)$ is a singular solution of \eqref{eq:Abel} then $u_{xx}(T,\tilde x)<0$.

\subsection{Number of limit cycles}
Developing {$u(1,x)$} in power series with respect to $x$ (see, e.g., \cite{ABF} or \cite{BFG}),
\[
\begin{split}
{u(1,x)} = & x + \left(\int_0^1 B(t)\,dt\right) x^2 + \left(\int_0^1 A(t)\,dt\right) x^3 
\\ & + \left(\int_0^1 A(t)\int_0^t B(s)\,ds\,dt\right) x^4 + \mathcal{O}\\
       = & x + \frac{3 t_B-1}{6} x^2 + \frac{2 - 3 t_A}{6} x^3 + \frac{ -16 + 21 t_A + 54 t_B - 75 t_A t_B}{360} x^4 + \mathcal{O},
\end{split}
\]
where $\mathcal{O}$ denotes higher order terms in $x$, $t_B-1/3$, and $t_A-2/3$.

In particular, if $t_A=2/3$ and $t_B=1/3$ then $u(t,x)-x=-x^4/540+\mathcal{O}(x^5)$, while 
the signs of the coefficients of $x^2$ and $x^3$ depend on $t_B$ and $t_A$ respectively. Hence, there is a double Hopf bifurcation of the origin giving rise to two positive closed solutions for $t_A<2/3$ and $t_B<1/3$.

To prove that the maximum number of positive closed solutions is two, 
we shall apply Theorem~\ref{theo:bifrotated}.  To this end, let us think of $-t_A$ as parameter $\lambda \in (\lambda_1, \lambda_2)$, with $\lambda_1 = -1$ and $\lambda_2 = -t_B$, so that
\[F(t,x,\lambda)=t(t+\lambda)x^3+(t-t_B)(t-1)x^2,\] and
$F_\lambda(t,x,\lambda) = t x^3>0,$ for all $t\in(0,1)$ and $x>0$.

For $\lambda=\lambda_2$, as was mentioned at the beginning of the section,
 \eqref{eq:Abelquadquad} has at most one positive closed solution.
Since for $t_B\neq 1/3$ the stability of the origin does not change, Theorem~\ref{theo:bifrotated} implies that 
\eqref{eq:Abelquadquad} has at most two positive closed solutions for every $t_B\neq 1/3$ and $t_B<t_A<1$. 

To conclude, note that $F_\lambda(t,x,\lambda)$ is  monotonic with respect to $-t_B$, 
and that singular positive closed solutions have the semistability given by Theorem~\ref{theo:bif},
so that if equation \eqref{eq:Abelquadquad} has more than two positive closed solutions for $t_B=1/3$ and some $t_B<t_A<1$ then
a small perturbation of $t_B$ would keep or increase that number of 
positive closed solutions, in contradiction with the maximum of two positive closed solutions 
for $t_B\neq 1/3$.

\section{Linear Trigonometric Coefficients}\label{Sect5}
Consider the Abel equation \eqref{eq:Abel_trigonometric_linear}, i.e.,
$x'=A(t)x^3+B(t)x^2$, where
\[
A(t)=a_0+a_1\sin t+a_2\cos t\quad \text{and}\quad B(t)=b_0+b_1\sin t+b_2\cos t,
\]
with $a_i,b_i\in\mathbb{R},\ i=0,1,2$. We prove that
Theorem~\ref{theo:bif} holds in a region where two positive limit cycles bifurcate from the origin,
obtaining an upper bound of two positive limit cycles.

Equation \eqref{eq:Abel_trigonometric_linear} has at most one simple positive limit cycle when $A$ or $B$ 
have definite sign~\cite{Pliss,GL},
when there is a linear combination of $A,B$ having definite sign~\cite{AGG}, 
or when the coefficients $a_i,b_i$, $i=0,1,2$, belong to certain regions~\cite[Theorem 1.2]{BFG}, in particular 
when $a_0 b_0=0$.
Note that the first condition corresponds to $A$ or $B$ having at most one zero in $[0,2\pi)$. Let us check that
the second condition holds either whenever
$A$ or $B$ have at most one zero in $[0,2\pi)$ or when there is no zero of $A$ (resp.~$B$) 
between the two zeros of $B$ (resp.~$A$). Therefore, we may assume that   $A,B$ have exactly two simple zeros in $[0,2\pi)$ 
which are interleaved.

\begin{prop}
Assume that $A$ or $B$ have at most one zero in $[0,2\pi)$, or that there is no zero of $A$ (resp.~$B$) 
between the two zeros of $B$ (resp.~$A$). Then there exist $\alpha,\beta\in\mathbb{R}$ 
such that $\alpha A(t)+\beta B(t)\geq 0$ for all $t\in\mathbb{R}$.
\end{prop}

\begin{proof}
If $A$ has at most one zero in $[0,2\pi)$ then $A(t)$ has definite sign in $[0,2\pi)$, 
so the result follows by choosing $\alpha=\pm 1$ and $\beta =0$. The same argument applies if $B$ has at most one zero in $[0,2\pi)$.

Assume now that $A,B$ have two zeros in $(0,2\pi)$, and that there is no zero of $B$ between the zeros of $A$
(with the other case being analogous). The change of variables $z\to \tan(t/2)$ transforms $A,B$
into rational functions with denominator $1+z^2$ and numerator a second degree polynomial. Moreover, the 
relative position of the roots of $A,B$ are preserved.
So we may assume that, after the change of variables, $A$ and $B$ become 
$\bar A(z)=a(z-z_{A_1})(z-z_{A_2})$ and $\bar B(z)=b(z-z_{B_1})(z-z_{B_2})$, respectively. Then  $\alpha \bar A+\bar B$ has definite sign if and only if  $d(\alpha):=\disc(\alpha \bar A+\bar B)\leq0$. Since $d(\alpha)$ is a degree two polynomial in $\alpha$ with positive leading coefficient, there exists $\alpha$ such that $d(\alpha)\leq0$ if and only if $\disc(d(\alpha))\geq0$. From
$$
\disc(d(\alpha))=16 a^2 b^2 (z_{A_1} - z_{B_1}) (z_{A_2} - z_{B_1}) (z_{A_1} - z_{B_2}) (z_{A_2} - z_{B_2}),
$$
we conclude. 
\end{proof}

From Remark \ref{rem:signos}, we  may assume that $A(0)=0$, $A'(0) <0 $, and $B(0) > 0$ since the remaining cases are similarly studied. Moreover, rescaling $x$, it is not restrictive to assume $A'(0)=-1$. Hence, in what follows we shall consider the equation 
\begin{equation}\label{eq:Abellinear2}
x'=\left(a_0-\sin t-a_0\cos t\right)x^3 + \left(b_0+b_1\sin t+b_2\cos t  \right)x^2,
\end{equation}
where $b_0+b_2>0$. 
Developing the solution of \eqref{eq:Abellinear2} in power series, we obtain
\begin{equation}\label{expandsolution}
\begin{split}
{u(2\pi,x)}=&x+2 b_0 \pi x^2 
+ (2 a_0 \pi + 4 b_0^2 \pi^2) x^3 \\
&+ \pi (3 a_0 b_1 - b_2 + 8 b_0^3 \pi^2 + 2 b_0 (1 + 5 a_0 \pi)) x^4+\mathcal{O}(x^5).
\end{split}
\end{equation}
In particular, there is a change of stability when $b_0=0$ or $a_0=0$, 
whereas when $a_0=b_0=0$ we have that $u(t,x)<x$ for $x>0$ close to the origin, which implies
that at least two limit cycles bifurcate from the origin with $a_0>0$ and $b_0<0$.

Actually, when $a_0=b_0=0$, $u(2\pi,x)<x$ for $x>0$ whenever $u(2\pi,x)$ is defined, as the following result establishes.

\begin{prop}{\cite[Theorem~2.4]{BT}}\label{prop:nosolutions}
For $a_0=b_0=0$, the Abel equation~\eqref{eq:Abellinear2}
has no positive limit cycles. Moreover, $u(2\pi,x)<x$
for any $x>0$ such that $u(t,x)$ is defined for $t\in[0,2\pi]$. 
\end{prop}

Let $u(t,x,a_0,b_0)$ be the solution of~\eqref{eq:Abel_trigonometric_linear}
determined by $u(0,x,a_0,b_0)=x$.
Note that the family \eqref{eq:Abellinear2} is monotonic with respect to both $a_0$ and $b_0$, so that the same holds for $u(t,x,a_0,b_0)$. In particular, we obtain the following result.  

\begin{coro}
If $a_0,b_0<0$ then \eqref{eq:Abellinear2} has no positive 
limit cycles. 
\end{coro}

Now we verify that \eqref{eq:Abel_trigonometric_linear} satisfies the hypotheses of Theorem~\ref{theo:bif}
for $a_0,b_0$ close to zero. 

In the next subsection, we will show that, while Proposition~\ref{prop:upper} can be used to show that $(C_2)$ holds for certain values of $b_1,b_2$, the hypotheses of Proposition \ref{PropC3} do not hold completely in this case. The reason is that both of these propositions much be verified on $[0,2\pi]\times [0,+\infty)$, while conditions $(C_2),(C_3)$ only need to be satisfied for the singular closed solutions. To avoid this problem, we use continuity arguments, studying the behaviour of the solutions at infinity in order to bound the region where there might be singular closed solutions.

\begin{prop}
  For each $b_1,b_2$ with $b_2>0$ there exists a neighbourhood of $(a_0,b_0)=(0,0)$ such that \eqref{eq:Abellinear2} satisfies $(C_1)$, $(C_2)$, and $(C_3)$ for any singular positive closed solution.
\end{prop}	
\begin{proof}
  Fix a neighbourhood $U=[-\epsilon_1,\epsilon_1]\times [-\epsilon_2,\epsilon_2]$ of $(a_0,b_0)=(0,0)$ in which the functions $A$ and $B$ have two zeros in $[0,2 \pi)$ and are interleaved. So, for any $b_1,b_2\neq 0$ and any $a_0,b_0$ in that neighbourhood, $(C_1)$ holds.

\medskip 
To prove $(C_2)$,
note that  $B$ has a simple zero in each of the intervals $(0,t_A)$ and $(t_A,2\pi]$. Thus, we can choose  $\delta_0>0$ 
so that,
for any positive smooth function $w$  satisfying 
\[
|w(t)|,|w'(t)| < \delta_0\quad \text{for all } 
t\in[0,2\pi],
\]
the function $2Aw+B$ also has a simple zero in each of the intervals $(0,t_A)$ and $(t_A,2\pi]$.

Denote
\[
u_\infty(t,a_0,b_0)=\sup \{u(t,x,a_0,b_0)\colon u(\cdot,x,a_0,b_0)\text{ being bounded in }[0,2\pi]\}.\]

Then $u_\infty(\cdot,a_0,b_0)$ is well-defined 
except for certain values of $t$ where the supremum is infinite. Moreover,  it is a solution of~\eqref{eq:Abel_trigonometric_linear} in each interval 
where it is defined. 
We will show in  Appendix~A that $u_\infty(t,a_0,b_0)$ is defined and is continuous for $t > 0$ and $(a_0,b_0)$
in a neighbourhood of $(0,0)$, and that
\[
\lim_{t\to\infty} u_\infty(t,0,0)=0. 
\]
Therefore,  there exists $n\in\mathbb{N}$ and a neighbourhood $U$ of $(0,0)$ such that $u_\infty(t+2\pi n,a_0,b_0)<\delta_0$ for all $t\in[0,2\pi]$ and $(a_0,b_0)\in U$, and 
\[
|\left(a_0-\sin t-a_0\cos t\right)x^3 + \left(b_0+b_1\sin t+b_2\cos t  \right)x^2|<\delta_0
\]
for all $t\in[0,2\pi]$, $x\in[0,u_\infty(t+2\pi n,x_0,a_0,b_0)]$, and $(a_0,b_0)\in U$.

\medskip

Let $u(t,\tilde x,a_0,b_0)$ be any singular positive closed solution  of~\eqref{eq:Abel_trigonometric_linear} with $(a_0,b_0)\in U$. Then  
$ u(t,\tilde x,a_0,b_0)= u(t+2\pi n,\tilde x,a_0,b_0)<u_\infty(t,a_0,b_0)<\delta_0$, and condition $(C_2)$ holds.

\medskip 

The last step is to prove that $(C_3)$ holds. It suffices to show that, for each $b_1,b_2$, with $b_2>0$, there exists a neighbourhood of $(a_0,b_0)=(0,0)$ such that 
the graph of any singular positive closed solution $\tilde u$ is disjoint with $v^{-1}(0)$. Hence,  the sign of $v(t,\tilde u(t,x))$ does not change
and is the same as the sign of $A'(0)B(0)$ (negative in this case). 

Since $v(t,0) = 4(B (t)A '(t)-B '(t)A (t))$, then $v(t,0)= -4b_2<0$ for $a_0=b_0=0$.
Making $\delta_0$ and $U$ smaller if necessary, we have	
 $v(t,x)<0$ for any $0<x\leq \delta_0$ and $(a_0,b_0)\in U$. To conclude,
it suffices to prove that  every 
singular positive closed solution $\tilde u(t,x)$ satisfies $\tilde u(t,x)\leq \delta_0$.
But that holds by the previous discussion, so there exists a neighbourhood of $a_0=b_0=0$ such that $(C_3)$ holds. 
\end{proof}

Let us prove that the maximum number of  positive closed solutions is two in a neighbourhood of $(a_0,b_0)=(0,0)$ in the quadrant where the double Hopf bifurcation occurs. 

\begin{theo}\label{theo:doslineal}
Assume there exists $\epsilon>0$ such that for every $-\epsilon<b_0<0<a_0<\epsilon$, \eqref{eq:Abellinear2} satisfies $(C_1)$, $(C_2)$,  and $(C_3)$ for any singular positive closed solution. 
Then \eqref{eq:Abellinear2} has at most two positive closed solutions for every $-\epsilon<b_0<0<a_0<\epsilon$.
\end{theo}
\begin{proof}
By Theorem~\ref{theo:bif}, $ u_{xx}(t,\tilde x)<0$ for every  singular positive closed solution~$ u(t,\tilde x)$. 

Let $$F(t,x)=\left(a_0-\sin t-a_0\cos t\right)x^3 + \left(b_0+b_1\sin t+b_2\cos t  \right)x^2.$$
The derivative of function $F$ respect to $b_0$ is strictly positive for all $t\in(0,2\pi)$, 
and there is at most one simple positive closed solution of \eqref{eq:Abellinear2} for any $a_0 > 0$ and $b_0=0$ (see \cite[Theorem 1.2]{BFG}).

Fix $0<a_0<\epsilon$. Then
Theorem~\ref{theo:bifrotated} implies that, for every $-\epsilon<b_0$, 
\eqref{eq:Abel_trigonometric_linear} has at most three positive closed solutions. Moreover,
if there are three positive limit cycles, one of them corresponds to a Hopf bifurcation 
at the origin and another to a Hopf bifurcation at infinity. 

From  {\eqref{expandsolution}}, we have that the origin is unstable for $b_0=0$. If there is a positive closed solution for $b_0=0$, it is 
stable, so that infinity is unstable, and therefore, for $b_0<0$, the infinity 
remains unstable, so that there is no bifurcation at infinity and the maximum number of positive closed solutions is two.
\end{proof}

If we could prove that \eqref{eq:Abellinear2} satisfies $(C_1)$,
$(C_2)$, and $(C_3)$ for any singular positive closed solution whenever
$A,B$ have two interleaved zeros, then a similar argument would prove
the upper bound of two positive closed solutions for
\eqref{eq:Abellinear2} with no additional conditions.

The main drawback of Theorem~\ref{theo:doslineal} is that there is no
clear way to estimate the value of $\epsilon$ for which the theorem
holds. A possible way of doing that would be to bound the homoclinic
connection at infinity, following~\cite{GGT} for instance, or
improving the results in Section~3. In the following, we shall
illustrate the limitations of these results.

\subsection{Algebraic computation of the conditions}

To conclude the section, we explore the applicability of Propositions~\ref{prop:upper} and \ref{PropC3} to Abel equation \eqref{eq:Abellinear2} with $b_0 + b_2 > 0$.

\medskip 

Proposition~\ref{prop:upper} applies in a certain region of the parameter space.
\begin{prop}\label{prop:C2trig}
Assume condition $(C_1)$ holds and let 
\[q = \frac{\sqrt{b_1^2+b_2^2}}{b_2^{1/3}}-2^{2/3}.\]
If $\vert a_0 \vert$ and $\vert b_0 \vert$ are small enough and either $q < 0$ or $b_2 > 2$, then condition $(C_2)$ holds. \end{prop}

\begin{figure}[h] 
\begin{center}
\begin{tikzpicture}[scale=1.5]
\clip (-3,-.5) rectangle (3,2.5);
\draw[help lines, color=gray!30, dashed] (-2.5,-.5) grid (2.5,2.5);
\draw[->,thick] (-2.5,0)--(2.5,0) node[below left]{$b_1$};
\draw[->,thick] (0,-.5)--(0,2.5) node[below left]{$b_2$};
\draw[smooth=0.7,ultra thick,oro!50!rojo] plot[domain=0:2,samples=200]({sqrt(2^(4/3)*\x^(2/3)-\x^2)},\x);
\draw[smooth=0.7,ultra thick,oro!50!rojo] plot[domain=0:2,samples=200](-{sqrt(2^(4/3)*\x^(2/3)-\x^2)},\x);
\draw[smooth=0.7,ultra thick,oro!50!rojo] plot[domain=-.1:.1,samples=200](\x,{2});
\end{tikzpicture}
\end{center}
\caption{Graph of $q=0$.}\label{fig:q}
\end{figure}
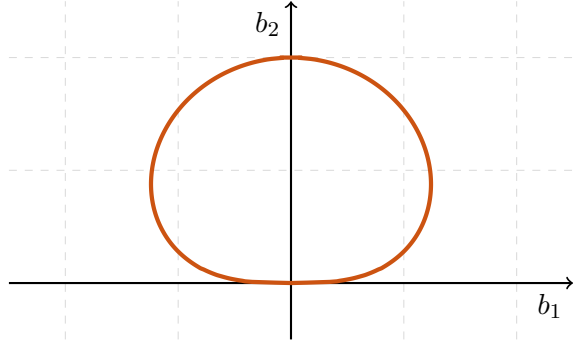

\medskip

Before proving Proposition~\ref{prop:C2trig}, we need some preliminary results.
By the change of variable $t=2\operatorname{atan}(z)+\pi$, we obtain that 
\[\bar A(z) := A(2\operatorname{atan}(z)+\pi) = \frac{2(z+a_0)}{z^2+1}\] and \[\bar B(z) :=  B(2\operatorname{atan}(z)+\pi)= \frac{(b_0+b_2) z^2 -2 b_1 z + b_0 - b_2}{z^2+1}.\] Notice that $\bar A(z)$ has only the zero $-a_0$ and that $\bar B(z)$ has the two zeros \[z^\pm_B := \frac{b_1 \pm \sqrt{b_1^2-(b_0^2-b_2^2)}}{b_0+b_2},\] and, since $2\operatorname{atan}(z)+\pi$ is strictly increasing, $b_0+b_2 > 0$, and the zeros of $A$ and $B$ are interleaved, we have that \[\bar B(-a_0) = \frac{(b_0+b_2)a_0^2 + 2b_1a_0+b_0-b_2}{a_0^2+1} < 0.\]

Moreover, one can easily check that this change of variable transforms
the function $P(t)$ defined in \eqref{ecu:v} into the rational
function $\bar P(z)$ such that $(z^2+1)^3 \bar P(z)$ is a polynomial
of degree six in $z$.

\medskip
Let $a_0 = b_0 = 0$. In this case, set $z_B := z= z_B^+ > 0$, so that \[b_1 = b_2 ( z_B^2-1)/(2 z_B),\quad z_B^- = -1/z_B.\] With this notation, $(z^2+1)^3 \bar P(z)_{\vert a_0=b_0 = 0} = (b_2/z_B^3)\, p(z)$ where
 \[
\begin{split}
p(z) = & (b_2^2-4) z_B^3 + 3 b_2^2 z_B^2 (z_B^2-1) z + 3 (-4 z_B^3 + 
   b_2^2 (z_B - 3 z_B^3 + z_B^5)) z^2 
   \\ & +b_2^2 (z_B^2-1)(z_B^4-8 z_B^2+1) z^3  -3 (4 z_B^3 + b_2^2 (z_B - 3 z_B^3 + z_B^5)) z^4
   \\& + 3 b_2^2 z_B^2 (z_B^2-1)z^5  -((b_2^2+4) z_B^3)z^6
\end{split}
\] 
The discriminant of $p(z)$ is \begin{equation}\label{discp}\Delta := 186624\, b_2^8(b_2^2+4)z_B^{15}(z_B^2+1)^{12} \big( b_2^4(z_B^2+1)^6-2^{10}\, z_B^6 \big).\end{equation} Since $z_B, b_2 > 0$, the sign of $\Delta$ is equal to the sign of \[b_2^{2/3}(z_B^2+1)-2^{5/3}\, z_B = 2 z_B \left(\frac{b_2^{2/3}(z_B^2+1)}{2 z_B} - 2^{2/3}\right)\] whose sign agrees with the sign of 
\[
\frac{b_2^{2/3}(z_B^2+1)}{2 z_B} - 2^{2/3} = \frac{\sqrt{b_1^2+b_2^2}}{b_2^{1/3}}-2^{2/3}=q.\]
Therefore the number of roots of $p(z)$ remains constant in each of the two connected regions determined by $q = 0$.

\begin{itemize}
\item For $b_1=0$ and $b_2=1$ we have that $q < 0$, and in this case $p(z) = -5z^6-9z^4-15z^2-3$ has no real roots. Hence $\bar P(z)_{\vert a_0=b_0 = 0}$ has no real zeros when $q < 0$, and the same holds for $\bar P(z)$ for $\vert a_0\vert, \vert b_0 \vert$ small enough and $q < 0$.
\item In the region where $q > 0,$ the polynomial $p(z)$ has exactly two real roots (take, for example, $b_1=1$ and $b_2 = 2$). Hence $\bar P(z)_{\vert a_0=b_0 = 0}$ has two real zeros when $q > 0$, and the same holds for $\bar P(z)$ for $\vert a_0\vert, \vert b_0 \vert$ small enough and $q > 0$.
\end{itemize}

\begin{lema}\label{lema4}
If $2 a_0 b_1 + (a_0^2+1) (b_0 + 2) < (1 - a_0^2) b_2$ and $\bar P(z)$ has two real zeros then they are separated by the zero of $\bar A(z)$.
\end{lema}

\begin{proof}
Since $\bar B(-a_0) < 0$, the leading coefficient of $(z^2+1)^3
\bar P(z)$ is equal to $-(b_0+b_2)((b_0+b_2)^2+4) < 0$, and \[\bar
P(-a_0) = \bar B(-a_0)(4-\bar B(-a_0)^2),\] we conclude that $\bar
P(-a_0) > 0$ if $2+\bar B(-a_0) < 0$, or equivalently $2 a_0 b_1 +
(a_0^2+1) (b_0 + 2) < (1 - a_0^2) b_2$. Therefore, if $\bar P(z)$ has
two real zeros and $2 a_0 b_1 + (a_0^2+1) (b_0 + 2) < (1 - a_0^2)
b_2$ then the two zeros of $\bar P(z)$ are separated by $-a_0$, which
is the zero of $\bar A(z)$.
\end{proof}

\begin{proof}[Proof of Proposition~\ref{prop:C2trig}]
On the one hand, if $q<0$ then the function $P$ has no zeros in $(0,t_{B_1})$ and $(t_A,t_{ B_2})$. By continuity, it also holds for $|a_0|,|b_0|$ small enough. 

On the other hand, since the limit of $2 a_0/(1 - a_0^2) b_1 + (a_0^2+1)(b_0 + 2)/(1-a_0^2)$ as $(a_0,b_0)$ tends to $(0,0)$ is $2$, if $b_2 > 2$, then $q > 0$ and there exist $a_0,b_0$ small enough such that $2 a_0 b_1 + (a_0^2+1) (b_0 + 2) < (1 - a_0^2) b_2$. In this case, by Lemma \ref{lema4}, $P$ has at most one zero in each of the intervals $(0,t_{B_1})$ and $(t_A,t_{ B_2})$.

In both cases, by Proposition \ref{prop:upper}, we conclude that $(C_2)$ is fulfilled. 
\end{proof}

\medskip 
Finally, let us show that Proposition \ref{PropC3} does not apply in the case of Abel equation \eqref{eq:Abellinear2} with $b_0 + b_2 > 0$.

\begin{prop}\label{prop:noC3}
For $a_0=b_0=0$ and every $b_1, b_2$, the set $v^{-1}(0) \cap \dot v^{-1}(0)$ is not empty.
\end{prop}

\begin{proof}
In this case, the function $Q(t)$ in Corollary \ref{CorC3} is equal to 
\[
3 b_2 (b_2 \sin(t) - b_1 \cos(t)).
\]
Thus $Q(t)$ has exactly two roots $t_1 < t_2 = t_1+\pi$ in $[0,2\pi)$. In particular, $t_1 \in [0,\pi)$ and $t_2 \in [\pi, 2\pi).$
Now, if we replace $b_1$ by $b_2 \tan(t_i)$ in $v(t_i,x),\ i = 1,2$, and solve the resulting quadratic equations in $x$, we obtain the following solutions: \[x_i^\pm = \frac{b_2 \pm \sqrt{b_2^2+4 \cos(t_i)^3}}{\sin(2 t_i)},\ i = 1,2.\]
Let us show that at least one of them is positive. For this, we distinguish two cases:
\begin{itemize}
\item If $t_1 \in (0,\pi/2)$ then $x_1^+ > 0$.
\item If $t_1 \in (\pi/2,\pi)$ then $t_2 \in (3\pi/2, 2 \pi)$ and $x_2^->0$. 
\end{itemize}
Hence, we conclude that $t_1$ or $t_2$ determines a real positive solution of $v(t_1, x) = 0$ or of $v(t_2,x) = 0$, respectively.

Now, substituting $b_1 =b_2 \tan(t_i),\ i = 1,2$, in both $v(t,x)$ and $\dot v(t,x)$, we obtain that \[\dot v(t_i,x) = x \left(\frac{b_2-\sin(2 t_i) x}{\cos(t_i)} \right) v(t_i,x),\ i = 1,2.\] Thus, by the previous argument, $v^{-1}(0) \cap \dot v^{-1}(0) \neq \emptyset$.
\end{proof}

 \appendix \label{apen}

\section{Stability at Infinity in the Trigonometric Case}

In this appendix, we study the stability at infinity following \cite[Subsection 3.1]{BFG2} adapted to our case. A first observation is that by the change of variables $y=1/x$, we have that \eqref{eq:Abellinear2} for $x>0$ is equivalent to 
\begin{equation}\label{eq:Abelinfinito}
y'=-B-A y^{-1}.
\end{equation}
Therefore, for $y>0$ the phase portrait of the integral curves of \eqref{eq:Abelinfinito} is the same as  the 
phase plane of the planar system
\begin{equation}\label{eq:sistemainfinito}
\begin{cases}
t'(s)&=y\\
y'(s)&=-B(t)y-A(t).
\end{cases}
\end{equation}
The equilibrium points in $(t,y)\in[0,2\pi)\times\mathbb{R}$, are $(0,0)$ and $(t_A,0)$. The linearization matrix at $(0,0)$ is 
$$
\left(\begin{array}{cc}
0 & 1\\1 & -(b_0+b_2)
\end{array}
\right)
$$
with eigenvalues $\lambda_-<0<\lambda_+$. 
Since $(0,0)$ is a saddle point, there exists a unique analytic invariant unstable manifold tangent to the line $\langle (1,\lambda_+)\rangle$ at $(0,0)$. The branch of the manifold in $\{(t,y): t\in[0,2\pi],\ y>0 \}$ is defined by the solution of \eqref{eq:sistemainfinito} that satisfies
\[
\lim_{s\to-\infty} (t(s),y(s))=(0,0),\quad \lim_{s\to-\infty}
\frac{y'(s)}{t'(s)}=\lambda_+.\quad
\]
Hence, there exists a unique analytic solution of \eqref{eq:Abelinfinito}, $v_{\infty}(t)$, defined in an interval $(0,\alpha)$ such that
 
\[
\lim_{t\to 0^+}v_\infty(t)=0,\quad \lim_{t\to 0^+} \frac{\partial v_\infty}{\partial t}(t)=\lambda_+.
\]
From the parametric unstable manifold theorem (see, e.g., {Theorem 5.13 of \cite{Chua}}),  this function is continuous with respect to  $a_0,b_0$.

 The linearization matrix at $(t_A,0)$ is 
 $$
 J(t_A,0)=\left(\begin{array}{cc}
 0 & 1\\-A'(t_A) & -B(t_A)
 \end{array}
 \right).
 $$
Taking into account that the trace and the determinant of $J(t_A,0)$ are $-B(t_A)>0$ and $A'(t_A)>0$ respectively, we have that $(t_A,0)$ is an unstable node or focus.

Since $(2\pi,0)$ is also a saddle point, there exists a unique analytic stable manifold tangent to the line $\langle (1,\lambda_-)\rangle$ at $(2\pi,0)$. The branch of the manifold in $\{(t,y): t\in[0,2\pi],\ y>0 \}$ is defined by the solution of \eqref{eq:sistemainfinito} that satisfies
\[
\lim_{s\to \infty} (t(s),y(s))=(2\pi,0),\quad \lim_{s\to \infty}
\frac{y'(s)}{t'(s)}=\lambda_-.\quad
\]
Thus, there exists a unique analytic solution of \eqref{eq:Abelinfinito}, $w_{\infty}(t)$, defined in an interval $(\beta,2\pi)$ such that
 
$$
\lim_{t\to 2\pi^-}w_\infty(t)=0,\quad \lim_{t\to 2\pi^-} \frac{\partial w_\infty}{\partial t}(t)=\lambda_-.
$$
Since the stable and unstable invariant manifolds are unique, $v_\infty(t)$ or $w_\infty(t)$ are defined in $(0,2\pi)$. Moreover, both are defined in $(0,2\pi)$ if and only if $v_\infty(t)=w_\infty(t)$. Note that the set of bounded solutions of \eqref{eq:Abellinear2} is limited by either $v_\infty(t)$ or $w_\infty(t)$.

We say $(a_0,b_0,b_1,b_2)$ is a bifurcation value at infinity if $v_\infty(t)=w_\infty(t)$. In this case a Poincar\'e map of \eqref{eq:Abelinfinito} is defined for $y>0$ close to zero as
$$
P(y)=v(2\pi,y),
$$
where $v(0,y)$ is the solution of \eqref{eq:Abelinfinito} such that $v(0,y)=y.$
The stability of the solution $v_\infty(t)=w_\infty(t)$ is established in the following result.
\begin{prop}\label{prop:estabilidadinfinito}
Assume $(a_0,b_0,b_1,b_2)$ is a bifurcation value at infinity. Then, for these values of the parameters, 
\[\sgn(P(y)-y)=-\sgn(b_0+b_2),\]
for every $y> 0$ small enough. 
\end{prop}

\begin{proof}

As the functions $A,B$ are $2\pi$-periodic, the system can be considered inside the cylinder. 
Then the solution $v_\infty$ is a homoclinic loop, and its stability is given by the sign of the trace of the linearization matrix at $(0,0)$ (see for instance  Section 10.3 of \cite{CH}).

\end{proof}

\begin{prop}
There exists a neighbourhood $U$  of $(a_0,b_0)=(0,0)$ such that, for each $(a_0,b_0)\in U$, $u_{\infty}(t,a_0,b_0)$ is defined for all $t>0$ and is continuous. Moreover,
$$
\lim_{t\to\infty}u_{\infty}(t,0,0)=0.
$$	
\end{prop}
	
\begin{proof}
As a direct consequence of Proposition~\ref{prop:estabilidadinfinito}
 for $a_0=b_0=0$ there is no periodic solution at infinity, 
otherwise the stability would be opposite to 
the stability of the origin, and consequently
there would be a limit cycle, in contradiction with Proposition~\ref{prop:nosolutions}.
In particular, $v_\infty$ is defined in $(0,\infty)$,  
$v_\infty(t)\to+\infty$ as $t\to \infty$, and is continuous with respect to $a_0,b_0$. Taking into account that $u_{\infty}=1/v_{\infty}$, we conclude.
\end{proof}

\section*{Acknowledgments}

The authors are partially supported by grant number IB18023 funded by Junta de Extremadura/ FEDER. The first two authors are also partially supported by grant number GR21056 funded by Junta de Extremadura/FEDER and by grant number PID2020-118726GB-I00 funded by MCIN/AEI/10. 13039/501100011033 and ``ERDF A way of making Europe''. The last author is also partially supported by grant GR21055 funded by Junta de Extremadura/FEDER and by grant number PID2022-138906NB-C21 funded by MCIN/AEI/10.13039/501100011033 and ``ERDF A way of making Europe''.

\section*{Data availability}

No data was used for the research described in the article.

\end{document}